\documentclass[12pt]{amsart}

\usepackage{latexsym}
\usepackage{mathrsfs}
\usepackage{amssymb}
\usepackage{amsfonts}
\usepackage{amsmath, amsthm}
\usepackage[all]{xy}
\usepackage{graphics}
\usepackage[english]{babel}

\newtheorem{lemma}{Lemma}
\newtheorem{proposition}{Proposition}

\newtheorem{theorem}{Theorem}
\theoremstyle{definition}
\newtheorem{remark}{Remark}
\theoremstyle{remark}
\newtheorem{eg}{Example}

\newcommand{\OS}{\Omega^{\mathrm{spin}}}
\newcommand{\z}{\mathbb Z}
\newcommand{\Q}{\mathbb Q}
\newcommand{\w}{\widetilde}
\newcommand{\Hom}{\mathrm{Hom}}
\newcommand{\sgn}{\mathrm{sign}}

\begin{document}

\title{On  $5$-manifolds with free fundamental group and simple boundary links in $S^5$}
\author{Matthias Kreck and Yang Su}

\begin{abstract}
We classify compact oriented $5$-manifolds with free fundamental group and $\pi_{2}$ a torsion free abelian group in terms of the second homotopy group considered as $\pi_1$-module, the cup product on the second cohomology of the universal covering, and the second Stiefel-Whitney class of the universal covering. We apply this to the classification of simple boundary links of $3$-spheres in $S^5$. Using this we give a complete algebraic picture of closed $5$-manifolds with free fundamental group and trivial second homology group.

\end{abstract}

\maketitle

\section{Introduction}

There is a close relation between classical links and closed $3$-manifolds since all $3$-manifolds are obtained by surgeries on links and Kirby calculus determines when two links give the same $3$-manifold. We consider a special case of such a relation in dimension $5$. The special condition on the side of links is that we only consider {\em simple boundary links} $L$ of a disjoint union of $3$-spheres in $S^5$, which means that the fundamental group of the complement is freely generated by the meridians of the link components. As in dimension $3$ we can perform surgery on the link $L$ to obtain a closed smooth manifold $M(L)$. It is easy to see that the fundamental group of $M(L)$ is a free group and $H_2(M(L);\mathbb Z)$ =0. In addition the second homotopy group is that of the complement $X$ of the link and this is torsion free as abelian group. One can ask which $5$-manifolds are obtained this way and for the classification of the links and the determination of the fibers of the map from links to $5$-manifolds given by surgery.

In this paper we answer this question by giving a classification of a more general class of closed $5$-manifolds, namely we classify all $5$-manifolds $M$ with $\pi_{1}(M)$ a free group and $\pi_2(M)$ torsion free as abelian group, in terms of an invariant we call {\em generalized Milnor pairing}, since it is a generalization of the Milnor pairing for knots. We also consider compact manifolds with boundary the disjoint union of copies of $S^1\times S^3$ and  free fundamental group, such that the fundamental group is freely generated by the circles in the boundary, and, as before, $\pi_2(M)$ torsion free as abelian group. We also define a topological version of the generalized Milnor pairing, called topological generalized Milnor pairing and prove a corresponding result for topological manifolds.

A second well known class of examples are fibered $5$-manifolds $M$ over the circle with simply-connected fibre. These are in the image of the surgery construction above if and only if we have a fibered knot and $H_2(M;\mathbb Z) =0$. But in general fibered $5$-manifolds over the circle have non-trivial second homology. Thus our more general class of manifolds also occurs naturally. See Remark 1 and the Appendix for more on this class of manifolds.

\smallskip

To give a feeling for the generalized Milnor pairing we define it in a special case, where $M$ is spin. Then it is represented by the triple $(\pi_1(M), \pi_2(M), b_M \colon \pi_2(M)^* \times \pi_2(M)^* \to (H^1(B\pi_1M;\mathbb Q [\pi_1M]))^*)$, where $b_M$ is given by the cup product. For details we refer to section 2. Now we formulate our main result.

\begin{theorem}\label{main}
Let $M_{0}$ and $M_{1}$ be two smooth (or topological), compact, oriented,  $5$-manifolds with free fundamental group of rank $n$ and torsion free $\pi_{2}$, with empty boundary or boundary consisting of $n$ copies of $S^1 \times S^3$ such that the circles in the boundary generate $\pi_1(M_i)$. Then there is an orientation preserving diffeomorphism (homeomorphism) between $M_{0}$ and $M_{1}$ if and only if there is an isomorphism between their (topological) generalized Milnor pairings.
\end{theorem}

We actually prove a stronger result about the realization of isomorphisms between the generalized Milnor pairings (Theorem 4).

\smallskip

Levine has classified $3$-dimensional simple knots in $S^5$ in terms of $S$-equivalence class of Seifert matrices (\cite{Levine}) and Liang has extended this to higher dimensional simple boundary links in terms of $l$-equivalence class of Seifert matrices (\cite{Liang} ). The general case of $3$-dimensional simple boundary links in $S^5$ seems to be open. Our classification result implies that Liang's result extends to dimension $3$.  Also by extending Liang's argument in higher dimension we can characterize the Seifert matrices occurring from links. We call the corresponding conditions {\em unimodularity conditions}.
Thus we obtain a complete algebraic picture of  simple boundary links in $S^5$.

\begin{theorem} \label{theorem: link}
The $l$-equivalence classes of Seifert matrices of simple boundary links of $3$-spheres in $S^5$ determine the isotopy type of the link. Moreover, the $l$-equivalence class of Seifert matrices gives a bijection from the set of isotopy classes of simple boundary links of $3$-spheres in $S^5$ to the set of $l$-equivalence classes of  square integral matrices $D$ satisfying the unimodularity conditions.
\end{theorem}

We would also like to give an algebraic picture of our closed $5$-manifolds. In general we don't know which values the generalized Milnor pairing takes. But if we require that $H_2(M;\mathbb Z)=0$, these manifolds are all results of surgeries on links and we can use the realization of the link invariants to give a complete answer.

Let $D$ be an $m \times m$ integral matrix satisfying the unimodularity conditions, then there is associated to $D$ a  $\z[F_{n}]$-module map $\varphi_{D} \colon (\z[F_{n}])^{m} \to (\z[F_{n}])^{m}$ and a generalized Milnor pairing
$$(F_{n}, \mathrm{coker} \varphi_{D}, b_{D} \colon (\mathrm{coker} \varphi_{D})^{*} \times (\mathrm{coker} \varphi_{D})^{*} \to (H^{1}(BF_{n}; \mathbb Q[F_{n}]))^{*})$$
We will give a detailed description of this in section 4.

\begin{theorem} \label{theorem: action}
There is a bijection between the diffeomorphism classes of closed oriented $5$-manifolds $M$ with $\pi_1(M)$ a free group of rank $n$ and $H_2(M;\z) =0$, and the isomorphism classes of generalized Milnor pairings $(F_n, \mathrm{coker} \varphi_{D}, b_{D})$ for all matrices D (with various size $m$) fulfilling the unimodularity conditions.
\end{theorem}

We will give more details of the generalized Milnor paring in section 2, and prove the main classification theorem in section 3. The discussion of $3$-links and their relation with $5$-manifolds will be the contents of section 4.

\smallskip

\begin{remark}
A special case of Theorem \ref{main1}   is when $\pi_{1}(M) \cong \z $ and $\pi_{2}(M)$ is a finitely generated abelian group. In this case we can show that $\pi_{2}(M)$ is torsion free and the bilinear form on $\pi_{2}(M)$ is unimodular, $w_{2}(\w M)$ is determined by the bilinear form on $\pi_{2}(M)$, and the realization problem of the invariants can be solved.  This gives a complete classification of closed $5$-manifolds with $\pi_{1} =\z$ and $\pi_{2}$ a finitely generated abelian group. As an application, this reproves the fibration theorems in dimension $5$ in the topological and smooth category given by \cite{Hsu}, \cite{Weinberger} and \cite{Shaneson} respectively. See more details in the Appendix.
\end{remark}

\begin{remark}
The notions of \emph{Borel manifolds} and \emph{strongly Borel manifolds} were coined in \cite[Definition 0.2]{KL}. A manifold $M$ is called a Borel manifold if for any homotopy equivalence $f \colon N \to M$ there exists a homeomorphism $h \colon N \to M$ such that $f$ and $h$ induce the same map on the fundamental groups up to conjugation. It is called strongly Borel if all homotopy equivalences are homotopic to a homeomorphism. If $M^{5}$ is a closed oriented spin topological $5$-manifold with free fundamental group and torsion free $\pi_{2}$, then it is Borel. Since for any homotopy equivalence $f \colon N^{5} \stackrel{\simeq}{\rightarrow} M^{5}$, $f$ induces an isomorphism between the topological generalized Milnor pairings (in this case the Kirby-Siebenmann invariant is determined by the bilinear form $b_{M}$, see the proof of Theorem \ref{main1}), the statement follows from Theorem 1. On the other hand, for a closed oriented topological $5$-manifold $M^{5}$ with free fundamental group, a computation of the topological structure set of $M$ using the surgery exact sequence gives $\mathscr S^{\mathrm{TOP}}(M^{5})=H^{2}(M;\z/2)$. Therefore by \cite[Theorem 1.1]{KL} $M$ is strongly Borel if and only if $H_{2}(M;\z/2)=0$.
\end{remark}

One often hears the statement, that the classification of high dimensional manifolds is completely understood. What people mean is that with the $s$-cobordism theorem one has a criterion when two manifolds are diffeomorphic and with surgery theory one has a reduction of the problem to find an $s$-cobordism to problems in homotopy theory (unstable and stable) and algebra (surgery obstruction groups) and the analysis of certain maps relating the homotopy theory and the algebra. But this doesn't mean that even for some very explicit manifolds like for example complete intersections the procedure can be carried out successfully. Given the complications of the homotopy groups of spheres, in higher dimensions the problems get harder and harder. But in comparatively low dimensions (say up to 8) one has a chance, which doesn't mean that it is routine. Most results in that dimension range concern simply connected manifolds. In this paper we make a first step towards a classification of 5-manifolds with fundamental group the free group $F_n$. This class is particular interesting since such manifolds on the one hand occur from total spaces of bundles over the circle and on the other hand as fundmental groups of links of $3$-spheres in $S^5$. We classify both in the smooth and topological category. It might be interesting to note that the topological classification of $4$-manifolds with fundamental group the free group $F_n$ is for $n>1$ completely open. The question whether the group $F_n$ is good in the sense of Freedman-Quinn \cite{Freedman} is the key question for topological $4$-manifolds. If this is the case then one can use similar methods as in the present paper to attack the classification of $4$-manifolds with fundamental group $F_n$.

\smallskip

\emph{Acknowledgement.} The first author would like to thank the Mathematical Institute of the Chinese Academy of Sciences in Beijing and the Max-Planck-Institute for Mathematics in Bonn for their support while this research was carried out. The second author would like to thank the Max-Planck Institute for Mathematics in Bonn for a research visit in August and  September, 2015. Both authors would like to thank the referee and the editor for their suggestions to improve the paper.

\section{The generalized Milnor pairing and the statement of the main theorem}

Now we describe the generalized Milnor pairing which we use to classify our manifolds. First we give the general algebraic definition. A \emph{generalized Milnor pairing} is a quadruple $(\pi_{1}, \pi_{2}, b, w_{2})$ consisting of the following:

\begin{enumerate}
\item $\pi_{1}$ a free group of rank $n$; let $\Lambda=\z[\pi_{1}]$ be the integral group ring and $\Lambda_{\mathbb Q}=\mathbb Q[\pi_1]$ be the rational group ring;
\item $\pi_{2}$ a finitely generate $\Lambda$-module, which is torsion free as an abelian group;
\item $b \colon \pi_{2}^{*} \times \pi_{2}^{*} \to (H^{1}(B\pi_{1};\Lambda_{\mathbb Q}))^{*}$  a symmetric $\Lambda$-equivariant pairing, where $*$ stands for the $\Q$-dual $\Hom_{\z}(-,\mathbb Q)$, and by $\Lambda$-equivariant we mean that under the diagonal action of $\Lambda$ on $\pi_{2}^{*} \times \pi_{2}^{*}$, and the natural $\Lambda$-module structure on $(H^{1}(B\pi_{1};\Lambda_{\mathbb Q}))^{*}$, $b$ is a $\Lambda$-module map;
\item $w_{2} \in \Hom(\pi_{2}, \z/2)$.
\end{enumerate}
An isomorphism $(\alpha, \beta) \colon (\pi_{1}, \pi_{2}, b ,w_{2}) \to (\pi_{1}', \pi_{2}', b' ,w_{2}')$ between generalized Milnor pairings consists of
\begin{enumerate}
\item an isomorphism $\alpha \colon \pi_{1} \to \pi_{1}'$;
\item an isomorphism $\beta \colon \pi_{2} \to \pi_{2}'$, which is compatible with the $\Lambda$- and $\Lambda'$-module structure, the pairings $b$ and $b'$, and mapping $w_{2}'$ to $w_{2}$.
\end{enumerate}

Let $M^{5}$ be a smooth closed oriented $5$-manifold with $\pi_1(M) \cong F_{n}$ and $\pi_{2}(M)$ a torsion free abelian group, we associate a generalized Milnor pairing $\varphi(M)=(\pi_{1}(M), \pi_{2}(M), b_{M}, w_{2}(\widetilde M))$ to $M$ as follows. Let $\widetilde M$ be the universal cover of $M$.  By Poincar\'e duality we have an isomorphism $H_4(\w M;\mathbb Q) = H_4(M;  \Lambda)\otimes \mathbb Q \cong H^1(M; \Lambda_{\mathbb Q})$ and the latter group is isomorphic to $H^{1}(B\pi_{1}(M); \Lambda_{\mathbb Q})$, because of the fact that $M$ has a CW-structure $M \simeq \vee_{n} S^{1} \vee \vee S^{2} \cup e^{3} \cdots$ (\cite[Proposition 3.3]{Wall1}). Next we use the Kronecker-isomorphism to identify $H^4(\w M;\mathbb Q)$ with $H_4(\w M; \mathbb Q)^*$, where $^{*}$ stands for the $\Q$-dual, and the isomorphism above to obtain an isomorphism $H^4(\w M; \mathbb Q) \cong (H^{1}(B\pi_{1}(M); \Lambda_{\mathbb Q}))^{*}$. The cup product and this identification together define a symmetric $\Lambda$-equivariant form $$H^2(\w M; \mathbb Q) \times H^2(\w M; \mathbb Q) \to (H^{1}(B\pi_{1}M;\Lambda_{\mathbb Q}))^{*}$$
Using the Kronecker isomorphism and the Hurewicz isomorphism we obtain a symmetric $\Lambda$-equivariant form
$$
b_M : \pi_2(M)^* \times \pi_2(M)^*  \to  (H^{1}(B\pi_{1}M;\Lambda_{\mathbb Q}))^{*}$$
where $^*$ is again the vector space of homomorphisms to $\mathbb Q$,  We will discuss more about this bilinear form in the beginning of section 3.

To this we add the second Stiefel Whitney class
$$w_2(\w M) \in \Hom(H_2(\w M;\mathbb Z), \mathbb Z/2)= \Hom (\pi_2(M), \mathbb Z/2)$$
 to obtain our invariant and get the quadruple
$$
\varphi(M)=(\pi_1(M), \pi_2(M), b_M, w_2(\w M)),
$$
we call this the {\em generalized Milnor pairing of $M$}. The group of self-isomorphisms of $\varphi(M)$ is denoted by $Aut(\varphi(M))$.

\begin{remark}
In the case where only spin manifolds are concerned, $w_{2}(\widetilde M)$ is always $0$, and the generalized Milnor pairing is actually a triple $\varphi(M)=(\pi_{1}(M), \pi_{2}(M), b_{M})$. This is the case in Theorem \ref{theorem: action}.
\end{remark}

\begin{remark}
It's easy to see from the Leray-Serre spectral sequence of the fibration $\w M \stackrel{p}{\rightarrow} M \to \vee_{n}S^{1}$ that $p^{*} \colon H^{2}(M;\z/2) \to H^{2}(\w M;\z/2)$ is injective. Therefore $w_{2}(M)$ and $w_{2}(\w M)$ determine each other.
\end{remark}

We also classify a special case of compact oriented manifolds $M$ with boundary which is relevant for classifying links in $S^5$. The boundary has to be a disjoint union of $n$ copies of $S^1 \times S^3$ and we require that the circles in the boundary components generate the fundamental group $F_n$ of $M$. Here we replace $H_4(\w M;\mathbb Q)$ by $H_4(\w M, \partial \w M;\mathbb Q)$ and we note that $H^2(\w M;\mathbb Q) \cong H^2(\w M, \partial \w M; \mathbb Q)$, so that the definition of $b_M$ makes sense. With this modification we can consider the quadruple defining $\varphi(M)$  as before. But we have to observe, that the identification of the fundamental groups of $M$ and $M'$ is now given by an identification of the boundary components.

\begin{remark}
When $X$ is the complement of a simple $3$-knot, then we have a bilinear paring $b \colon H^{2}(\w X ;\Q) \times H^{2}(\w X;\Q) \to \Q$, which  is the Milnor paring (\cite{Milnor}, see also \cite{Litherland}).
\end{remark}

We also classify the corresponding topological manifolds. Here we add a fifth term to our invariant, the Kirby-Siebenmann invariant $KS(M) \in H^{4}(M;\z/2) \cong \pi_1(M)/[\pi_{1}(M), \pi_{1}(M)] \otimes \mathbb Z/2$. We call the quintuple $(\pi_{1}(M), \pi_{2}(M), b_{M}, w_{2}(\widetilde M), KS(M))$ the \emph{topological generalized Milnor pairing} of the topological manifold $M$.  Of course in the definition of an isomorphism $(\alpha, \beta)$ between two topological generalized Milnor pairings we require that the isomorphism $\alpha \colon \pi_{1}(M) \to \pi_{1}(M')$ respects the Kirby-Siebenmann
invariant, too.

Now we restate the classification theorem of the manifolds under consideration and add the realization statement for induced maps.

\begin{theorem}\label{main1}
Let $M_{0}$ and $M_{1}$ be two smooth (or topological), closed, oriented, $5$-manifolds with free fundamental group of rank $n$ and torsion free $\pi_{2}$. Then $M_0$ and $M_1$ are oriented-diffeomorphic (-homeomorphic) if and only if their (topological) generalized Milnor pairings are isomorphic.
Any isomorphism between the (topological) generalized Milnor parings can be realized by an orientation-preserving diffeomorphism (homeomorphism) from $M_{0}$ to $M_{1}$.

If $M_0$ and $M_1$ are compact with boundary consisting of $n$ copies of $S^1 \times S^3$ such that the circles in the boundary generate $\pi_1(M_i)$. Then $M_0$ and $M_1$ are oriented-diffeomorphic (homeomorphic) if and only if there exists an isomorphism $(\alpha, \beta)$ between their (topological) generalized Milnor pairings, where $\alpha$ is induced by identifying the boundary components. Any such isomorphism can be realized by an orientation-preserving diffeomorphism (homeomorphism).
\end{theorem}

The isomorphism $\alpha$ above actually sends free generators $x_i$ of $\pi_1(M_0)$ to conjugates of free generators $x_i'$ of $\pi_1(M_1)$, which are represented by different arcs in the interior to a base-point.

\begin{remark}
Note that in the definition of the invariant $\varphi (M)$ we use the cup product on the cohomology with rational coefficients. Usually one loses information when passing from integral coefficient to rational coefficient. But in our situation, the rational cohomology contains essentially more information than the integral cohomology. This can be  illuminated by the following example.
\end{remark}

\begin{eg}
Let
$$A= \left ( \begin{array}{cc}
2 & 0 \\
3 & 1 \end{array} \right )$$
then  $A+A'$ ($A'$ is the transpose of $A$) is unimodular and has signature $0$. Therefore by \cite[Theorem 2]{Levine} there is a simple $3$-knot $K \subset S^{5}$ with Seifert matrix $S$-equivalent to $A$. The Alexander polynomial of $K$ is $\Delta_{K}(t)=\det(A-tA')=2t^{2}+5t+2$. Let $X$ be the complement of $K$, then by \cite[Theorem 1.5]{Crowell} $H_{2}(\w X) \cong \z[1/2] \oplus \z[1/2]$. Let $M^{5}$ be the result of surgery on $K$, then $\pi_{1}(M) \cong \z$, $\pi_{2}(M) \cong H_{2}(\w M) \cong H_{2}(\w X) \cong \z[1/2] \oplus \z[1/2]$. We see that $H^{2}(\w M;\z)=0$ but $H^{2}(\w M;\Q) \cong \Q^{2}$.
\end{eg}

\section{proof of  Theorem 4}
Before giving the proof of the main theorem we first rephrase the bilinear form $b_{M}$ in a more explicit form. Fix an identification $\pi_{1}(M) \stackrel{\cong}{\rightarrow} F_{n}$ and consider the classifying map of the fundamental group $f \colon M \to BF_{n}=\vee_{i=1}^{n} S^{1}_{i}$. From the Leray-Serre spectral sequence (with twisted coefficients, which we denote by underbar) of the fibration $\widetilde M \to M \to \vee_{i=1}^{n} S^{1}_{i}$ we get an isomorphism $H_{5}(M) \to H_{1}(\vee_{n}S^{1}; \underline{H_{4}(\widetilde M)})$. Note that
$H_{1}(\vee_{n} S^{1};\underline{H_{4}(\w M)})=\mathrm{Ker}(\oplus_{n}H_{4}(\w M) \stackrel{d}{\rightarrow} H_{4}(\w M))$,
where $d(x_{1}, \cdots ,x_{n}) = \sum_{i}(g_{i}-1)x_{i}$, with $g_{1}, \cdots, g_{n}$ the corresponding generators of $F_{n}$.
 This leads to an injection $H_{5}(M) \to \oplus_{n}H_{4}(\widetilde M)$. Denote the image of the fundamental class $[M]$ by $(\sigma_{1}, \cdots , \sigma_{n})$.  Now denote by $I_{i}(M)$ the symmetric bilinear form $H^{2}(\w M; \Q) \times H^{2}(\w M;\Q) \to \Q$, $I_{i}(\alpha, \beta)=\langle \alpha \cup \beta, \sigma_{i}\rangle$. From the relation $\sum_{i}(g_{i}-1)\sigma_{i}=0$ we see that the bilinear forms satisfy the relation $\sum_{i}I_{i}(\alpha, \beta)=\sum_{i}I_{i}(g_{i}^{*}\alpha, g_{i}^{*}\beta)$.

 Geometrically, we choose regular values $q_{i} \in S_{i}^{1}$  and denote $F_{i}=f^{-1}(q_{i})$. Let $E$ be the complement of an open tubular neighborhood of $\cup_{i} F_{i}$, then $E$ has boundary $\partial E = \cup_{i}F_{i}^{\pm}$, where $F_{i}^{\pm}$ is the positive and negative boundary component of the tubular neighborhood of $F_{i}$.  $\w M$ is obtained by glueing infinitely many copies of $E$ under the deck transformation, i.~e.~$\w M = \cup_{g \in F_{n}} E_{g}$. Let $\overline M_{i} \to M$ be the $\z$-covering of $M$ corresponding to $M \to \vee_{i=1}^{n} S^{1}_{i} \to S^{1}_{i}$, then it's easy to see that the Leray-Serre spectral sequence of this covering gives an isomorphism $H_{5}(M) \stackrel{\cong}{\rightarrow} H_{4}(\overline M_{i})$, with $[M] \mapsto [F_{i}^{-}]$. Furthermore the commutative diagram
 $$\xymatrix{
 \w M \ar[dr] \ar[rr] & & \overline M_{i} \ar[dl] \\
 & M & \\}
$$
induces
$$\xymatrix{
& \oplus_{n}H_{4}(\w M) \ar[d]^{\textrm{projection to the $i$-th component }} \\
H_{5}(M) \ar[ur] \ar[dr] & H_{4}(\w M)_{i} \ar[d] \\
& H_{4}(\overline M_{i}) \\}
$$
From this we see that
each $\sigma_{i}$ is represented by $F_{i}^{-}$ in $E \subset \w M$.

By \cite[Proposition 3.3]{Wall1} we know that $M$ has a CW-structure of the form $M \simeq \vee_{i=1}^{n} S^{1}_{i} \vee \vee S^{2} \cup e^{3} \cdots$. Therefore we have isomorphisms  $H_{4}(\widetilde M) \cong H^{1}_{c}(\widetilde M) \cong H^{1}(\vee_{n}S^{1}, \Lambda)$, where $\Lambda$ denotes the group ring $\z [F_{n}]$. Thus we have a surjection $\Lambda^{n} \to H_{4}(\widetilde M)$. Let $e_{i}$ be the standard basis of $\Lambda^{n}$, then  $e_{i}$ is mapped to  $\sigma_{i}$. Therefore  $\sigma_{1}, \cdots , \sigma_{n}$ form a set of generators of the $\Lambda$-module $H_{4}(\widetilde M)$. For any $\alpha, \beta \in H^{2}(\w M;\Q)$, $x \in H_{4}(\w M)$, we may assume that $x = \sum_{i}\lambda_{i} \sigma_{i}$, with $\lambda_{i}=\sum_{g} a^{(i)}_{g} \cdot g \in \Lambda$. Then $\langle \alpha \cup \beta, x \rangle = \langle \alpha \cup \beta, \sum_{i}\lambda_{i} \sigma_{i} \rangle = \sum_{i,g} a_{g}^{(i)} \langle g^{-1} \alpha \cup g^{-1} \beta, \sigma_{i} \rangle = \sum_{i,g} a_{g}^{(i)} I_{i}(g^{-1} \alpha, g^{-1} \beta)$. Thus we have shown:

\begin{lemma}\label{form}
The sequence of bilinear forms $(I_{1}, \cdots, I_{n})$ contain the same information as the bilinear pairing $b_{M}$ together with an identification of $\pi_1(M)$ with the free group  $F_n$.
\end{lemma}

Next we relate the signature of forms $I_i$ to the signatures of the fibre $F_i^4$.

\begin{lemma}\label{signature}
The bilinear form $I_{i} \colon H^{2}(\widetilde M;\mathbb Q) \times H^{2}(\widetilde M;\mathbb Q) \to \mathbb Q$ has the same signature as the signature  of the intersection form of $F_{i}^{4}$.
\end{lemma}
\begin{proof}
We use homology and cohomology with $\mathbb Q$-coefficients.

Let  $E$ be the exterior of an open tubular neighborhood of $\cup_{i} F_{i}$. Then  the universal cover $\w M$ is
$\widetilde M = \cup_{g \in F_{n}} E_{g}$
where each $E_{g}$ is a copy of $E$. Since $H_{2}(F_{i})$ is finite dimensional, there exists a connected compact submanifold $M_{0} \subset \w M$, which is a union of finitely many $E_{g}$'s,  $F_{i} \subset M_{0}$, such that any $x \in \mathrm{Ker}(H_{2}(F_{i}) \to H_{2}(\widetilde M))$ is in $\mathrm{Ker}(H_{2}(F_{i}) \to H_{2}(M_{0}))$. Therefore $\mathrm{Ker}(H_{2}(F_{i}) \to H_{2}(\widetilde M))=\mathrm{Ker}(H_{2}(F_{i}) \to H_{2}(M_{0}))$. Dually on cohomology, we have
$$\mathrm{Im}(H^{2}(\widetilde M) \to H^{2}(F_{i}))=\mathrm{Im}(H^{2}(M_{0}) \to H^{2}(F_{i})).$$

$\partial M_{0}$ has a component $F_{0}$, which is the image of $F_{i}$ under a deck transformation by $g \in \pi_{1}(M)$. There is a commutative diagram
$$\xymatrix{
H^{2}(M_{0}) \ar[r] \ar[dr] & H^{2}(F_{i}) \ar[d]^{g^{*}}\\
& H^{2}(F_{0}) }
$$
where $g^{*}$ is an isometry. So we have
$$\begin{array}{ccl}
H^{2}(\w M)/\mathrm{rad(I_{i})} & = & \mathrm{Im}(H^{2}(\w M) \to H^{2}(F_{i}))/\mathrm{rad} \\
& = & \mathrm{Im}(H^{2}(M_{0}) \to H^{2}(F_{i}))/\mathrm{rad} \\
& \cong &  \mathrm{Im}(H^{2}(M_{0}) \to H^{2}(F_{0})) /\mathrm{rad}
\end{array}
$$

Note that $\mathrm{Ker}(H_{2}(\partial M_{0}) \to H_{2}(M_{0}))$ is a Lagrangian in $H_{2}(\partial M_{0})$. Therefore
$$\mathrm{Ker}(H_{2}(F_{0}) \to H_{2}(M_{0})) = \mathrm{Ker}(H_{2}(\partial M_{0}) \to H_{2}(M_{0}))  \cap H_{2}(F_{0})$$
is isotropic. A standard argument in linear algebra shows that dually on cohomology, $\mathrm{Im}(H^{2}(M_{0}) \to H^{2}(F_{0}))$ has a complement which is isotropic. Let's denote it by $K$, it generates a hyperbolic form $H(K)$ in $H^{2}(F_{0})$ and we have
$\mathrm{Im}(H^{2}(M_{0}) \to H^{2}(F_{0}))/\mathrm{rad} \oplus H(K) = H^{2}(F_{0})$.
Therefore $\mathrm{sign}(I_{i})=\mathrm{sign}(H^{2}(F_{0}))$.
\end{proof}

The proof of Theorem \ref{main1} is based on modified surgery theory. We refer to \cite{Kreck99} for the details of this machinery for classifying manifolds. For the convenience of the reader we summarize the basic concepts and the main theorem we apply. The basic idea is to weaken the normal homotopy type which is the first basic invariant of a manifold $M$ in classical surgery to the normal $k$-type. This is roughly given by the $k$-skeleton of $M$ together with the restriction of the normal bundle. Since the $k$-skeleton is not well defined we pass to Postnikov towers instead or better Moore-Postnikov decompositions. The normal bundle is equivalent to the normal Gauss map $\nu : M \to BO$. The {\bf normal $k$-type} is the $k$-th stage of the Moore Postnikov tower of $\bar \nu$, which is s fibration $p:B_k(M) \to BO$ which is completely characterized by the property that there is a lift $\bar \nu : M \to B_k(M) $ of $\nu$ which induces an isomorphism  on homotopy groups up to degree $k$ and is surjective in degree $k+1$. Note that if $k$ is larger than the dimension of $M$ the normal $k$-type is equivalent to the normal homotopy type, thus modified surgery generalizes classical surgery. Such a lift is called a {\bf normal $k$-smoothing}.

Given two normal $k$-smoothings $(M, \bar \nu_M)$ and $(M', \bar \nu_{M'})$ in the same fibration $B_k$ the first step is to decide whether these normal $k$-smoothings are bordant. This means that there is a coboundary $W$ together with a lift of the normal Gauss map $\bar \nu_W$ (but this is not highly connected). The main theorem of modified surgery is that if $k \ge  \frac {\dim M}{2}-1$, then there is a surgery obstruction in a monoid $l_{\dim M+1} (\pi_1(M),w_1(M))$ from which one can decide whether $W$ is $B_k$-bordant to an $s$-cobordism.

\smallskip

Now we return to our situation of $5$-manifolds. We will work with the normal $2$-type of $M$. Then the obstruction is actually in the classical Wall group $L_5(\pi_1(M),w_1(M))$. We prepare the proof with a construction of the normal $2$-type (c.~f.~\cite[Proposition 2]{Kreck99}) of  a smooth manifold $M$ (of arbitrary dimension) which might be of separate interest elsewhere.  Let $u \colon M \to P$ be the second stage Postnikov tower of $M$, there are unique cohomology classes $w_{i} \in H^{i}(P;\z/2)$ ($i=1,2$) such that $u^{*}(w_{i})=w_{i}(M)$. Let $w_{1} \times w_{2} \colon P \to K(\z/2,1) \times K(\z/2,2)$ be the classifying map of these classes, and $w_{1}(EO) \times w_{2}(EO) \colon BO \to  K(\z/2,1) \times K(\z/2,2)$ be the classifying map of the universal Stiefel-Whitney classes. Consider the following pullback square
$$\xymatrix{
B(\pi_{1}(M), \pi_{2}(M), k_{1}, w_{1}(M), w_{2}(M)) \ar[r]^{\ \ \ \ \ \ \ \ \ \ \ \ \ h} \ar[d]^{p} & P \ar[d]^{w_{1} \times w_{2}} \\
BO \ar[r]^{w_{1}(EO)\times w_{2}(EO) \ \ \ \ \ \ \  \ }  & K(\z/2,1) \times K(\z/2,2) }
$$
then there is a lift $\overline{\nu} \colon M \to B(\pi_{1}(M), \pi_{2}(M), k_{1}, w_{1}(M), w_{2}(M))$ of the normal Gauss map $\nu \colon M \to BO$ of $M$,  which a $3$-equivalence, and $p$ is $3$-coconnected. Thus we have shown:

\begin{lemma} The fibration
$$p \colon B(\pi_{1}(M), \pi_{2}(M), k_{1}, w_{1}(M), w_{2}(M)) \to BO$$
 is the normal $2$-type of $M$. There is a corresponding construction in the topological category, if one replaces $BO$ by $BTop$.
\end{lemma}

Now we are ready to prove Theorem \ref{main1}.

\begin{proof}[Proof of Theorem \ref{main1}] We begin with the smooth category. In our situation, the second stage Postnikov tower $P$ of $M$  is a fibration over $\vee_{i=1}^{n}S^{1}_{i}$ with fiber $K=K(\pi_{2}(M),2)$, and monodromy given by the $\pi_{1}(M)$-module structure of $\pi_{2}(M)$. We denote the normal $2$-type by  $p \colon B \to BO$, and recall that by the lemma above it is determined by $\pi_{1}(M)$, $\pi_{2}(M)$ as a $\z[\pi_{1}(M)]$-module, and $w_{2}(M)$.

Now we compute the bordism group $\Omega_{5}(B,p)$. Note that $\Omega_{5}(B,p)=\pi_{5}^{S}(M(p))$. We consider the fibration $\w B \to B \to \vee_{i=1}^{n} S^{1}_{i}$, the Wang sequence of the generalized homology theory $\pi_{*}^{S}$ is
$$\cdots \to \Omega_{5}(\w B, \w p) \to \Omega_{5}(B,p) \to \oplus_{n}\Omega_{4}(\w B, \w p) \to \cdots$$
where $\w B $ is the pullback
$$\xymatrix{
\w B \ar[r] \ar[d]^{\w p} & K \ar[d]^{\mathrm{const} \times w_{2}} \\
BO \ar[r] & K(\z/2,1) \times K(\z/2,2) }
$$
where $w_{2}\in H^{2}(K;\z/2)$ is the image of $w_{2} \in H^{2}(P;\z/2)$ under the injection $H^{2}(P;\z/2) \to H^{2}(K;\z/2)$. From this we have $\Omega_{n}(\w B, \w p)=\OS_{n}(K;\eta)$, where the latter group is the bordism group of $f \colon M \to K$ together with a spin structure on $f^{*}\eta \oplus \nu M$, where $\eta$ is a complex line bundle over $K$ such that $w_{2}(\eta)=w_{2} \in H^{2}(K;\z/2)$.

$\pi_{2}(M)$  is the direct limit of its finitely generated subgroups, by assumption $\pi_{2}(M)$ is a torsion-free abelian group, hence it is a direct limit of finitely generated free abelian groups $\varinjlim G_{\alpha}$. Therefore $K$ is a direct limit of spaces $K=\varinjlim K(G_{\alpha}, 2)$.  In general there is an Atiyah-Hirzebruch spectral sequence computing $\OS_{n}(X;\eta)$ with $E_{2}$-terms $H_p(X;\OS_q)$ and the differential $d_2$  is dual to $Sq^2+w_2(\eta)\cdot$ (\cite{Teichner}). An easy computation with this spectral sequence shows that  $\OS_{5}(K(G_{\alpha},2);\eta)=0$ for a finitely generated free abelian group $G_{\alpha}$, and henceforth $\OS_{5}(K;\eta)= \varinjlim \OS_{5}(K(G_{\alpha}, 2);\eta)=0$.

Therefore we have an injection $\Omega_{5}(B,p) \to \oplus_{n} \OS_{4}(K;\eta)$. There is a commutative diagram
$$\xymatrix{
\Omega_{5}(B,p) \ar[r] \ar[d] & \oplus_{n} \OS_{4}(K;\eta) \ar[d] \\
H_{5}(P)  \ar[r] & \oplus_{n} H_{4}(K)}$$
with the horizontal arrows injective
and the vertical arrows the edge homomorphisms.  Following the definition of the boundary map in the Mayer-Vietoris sequence of the bordism theory, we see that a bordism class $[f \colon M \to B]$ is mapped to
$$([h\circ f \colon F_{1} \to K], \cdots, [h \circ f \colon F_{n} \to K]) \in \oplus_{n} \OS_{4}(K;\eta)$$
where $h \colon B \to P$ is the map in the pullback square, $\pi \colon P \to \vee_{n} S^{1}$ is the projection map, and $F_{i}=(\pi \circ h \circ f)^{-1}(q_{i})$ is the preimage of a regular value $q_{i} \in S^{1}_{i}$.  A direct calculation with the Atiyah-Hirzebruch spectral sequence shows that a bordism class  $[\varphi \colon N^{4} \to K] \in \OS_{4}(K(G_{\alpha},2);\eta)$ is determined by $\mathrm{sign}(N)$ and $\varphi_{*}[N] \in H_{4}(K(G_{\alpha},2))$. Passing to the limit we see that a bordism class $[\varphi \colon N^{4} \to K] \in \OS_{4}(K;\eta)$ is determined by $\mathrm{sign}(N)$ and $\varphi_{*}[N] \in H_{4}(K)$. Now $H_{4}(K)=\varinjlim H_{4}(K(G_{\alpha},2))$
 is a direct limit of free abelian groups, hence is torsion-free, therefore $\varphi_{*}[N]$ is determined by its image in $H_{4}(K;\Q)$, which is further determined by the evaluation with elements in $H^{4}(K;\Q)$. Note that $H^{4}(K;\Q)=H^{4}(K(\pi_{2}(M)\otimes \Q,2);\Q)$ where $\pi_{2}(M)\otimes \Q$ is a $\Q$-vector space. From this it's easy to see that the cup product map $H^{2}(K;\Q) \otimes H^{2}(K;\Q) \stackrel{\cup}{\rightarrow} H^{4}(K;\Q)$ is surjective, therefore $\varphi_{*}[N] \in H_{4}(K;\Q)$ is determined by $\langle \varphi^{*}\alpha \cup \varphi^{*}\beta, [N]\rangle $, $\forall \alpha, \beta \in H^{2}(K;\Q)$.

For a normal $2$-smoothing $\overline{\nu} \colon M \to B$, let $f \colon M \stackrel{\overline{\nu}}{\rightarrow}B \stackrel{h}{\rightarrow} P$ be the composition, we have a commutative diagram
 $$\xymatrix{
 \w M \ar[d] \ar[r]^{\w f} & K \ar[d] \\
 M \ar[r]^{f} & P \\}
 $$
and  $f \colon F_{i} \to K$ is the composition $\w f \circ i \colon F_{i} \subset \w M \to K$. Notice that $\w f^{*} \colon H^{2}(K;\Q) \to H^{2}(\w M;\Q)$ is an isomorphism, therefore the evaluation $\langle f^{*}\alpha \cup f^{*}\beta, [F_{i}] \rangle =\langle \w f^{*} \alpha \cup \w f^{*} \beta, i_{*}[F_{i}] \rangle =\langle \w f^{*} \alpha \cup \w f^{*} \beta, \sigma_{i} \rangle $ is exactly the bilinear form $I_{i} \colon H^{2}(\w M;\Q) \otimes H^{2}(\w M;\Q) \to \Q$.

By Lemma \ref{signature} $\sgn(F_{i})$ equals the signature of the bilinear form $I_{i}$. This shows that the bordism class $[M, \overline{\nu}]$ is determined by the bilinear forms $I_{i}$ ($i=1, \cdots , n$).

Now given two manifolds $M$ and $M'$ with isomorphic algebraic invariants and - depending on an ordering of the boundary components in the bounded case - equal boundary  as in Theorem \ref{main1}, then they have the same normal $2$-type $(B,p)$. We identify the boundaries (one of the manifolds with opposite orientation) to obtain a closed manifold and use the normal $2$-smoothings $\overline{\nu} \colon M \to B$ and $\overline{\nu'} \colon M' \to B$ to obtain an element in $\Omega_{5}(B,p)$. We note here that we controlled the restriction of the normal $2$-smoothings to the boundary by requiring that the identification of the boundary components is compatible with the identification of the fundamental groups. By the consideration above this is the zero element if our invariant $\varphi$ agrees for $M_0$ and $M_1$ with the normal $2$-smoothings chosen such that the invariants agree.

Let $W$ be a $B$-null-bordism of the glued manifold, then there is an obstruction $\theta(W) \in l_{6}(F_{n})$. If this is elementary, then $W$ is $B$-bordant rel.~boundary to an $s$-cobordism \cite[Theorem 3]{Kreck99}. In our situation with $\pi_{1}(M) \cong F_{n}$ the Whitehead group $Wh(F_{n})=\oplus Wh(\z)=0$, and so we won't have to consider the preferred bases. Furthermore by the remark on \cite[p.730]{Kreck99} the obstruction sits in the ordinary $L$-group $L_{6}(F_{n})$. This group is isomorphic to $\z/2$  and the obstruction is detected by the Arf-invariant (\cite[Theorem 16]{Cappell}). Since there is a simply-connected closed $6$-manifold with Arf-invariant $1$ we can change $W$ by disjoint sum with this, if necessary, to show that $\theta(W)=0 \in L_{6}(F_{n})$. This implies that $\theta(W)$ is elementary and finishes the proof  in the smooth case.

The proof of the topological case is similar, since the modified surgery method also applies to topological manifolds (c.~f.~\cite{Kreck99}). The only difference is that an element $ [\varphi \colon F^{4} \to K] \in \Omega_{4}^{\mathrm{TopSpin}}(K;\eta)$ is determined by the image of the fundamental class $\varphi_{*}[F] \in H_{4}(K)$, the signature $\mathrm{sign}(F)$ and the Kirby-Siebenmann invariant $KS(F)$. Each $F_{i}$ has trivial normal bundle in $M$, therefore under the isomorphism $H^{4}(M;\z/2) \stackrel{\cong}{\rightarrow} \oplus_{i=1}^{n} H^{4}(F_{i} ;\z/2)$, $KS(M)$ is mapped to
$$(KS(F_{1}), \cdots, KS(F_{n})).$$
The rest is the same as in the smooth case.
\end{proof}

\section{ Proof of Theorem 2 and 3}
The Seifert matrix of a boundary link is defined as follows (c.~f.~\cite{Liang}): choose Seifert manifolds $F_{i}$ of the link $L$, then there are linking forms
$$\theta_{ij} \colon H_{q}(F_{i}) \otimes H_{q}(F_{j}) \to \z, \ \ (\alpha, \beta) \mapsto L(z_{1}, z_{2})$$
defined by linking numbers between $z_{1}$, representing $\alpha$, and $z_{2}$, representing $i_{+}\beta$. With respect to a basis of the torsion-free part of $H_{q}(F_{i})$ the linking forms $\theta_{ij}$ are represented by a matrix $A_{ij}$, then the Seifert matrix $D=(A_{ij})$ of $L$ is an integral square matrix formed by the  blocks $A_{ij}$, and $D$ is $(-1)^{q}$-symmetric. Different choices of Seifert manifolds will lead to different Seifert matrices, but they are related by a sequence of ``algebraic moves'' and are \emph{$l$-equivalent}. The $l$-equivalence class of the Seifert matrix $D$ is a well-defined invariant of $L$ (\cite[Theorem 1]{Liang}).

Given a square integral matrix $D=(A_{ij})$, consisting of matrices blocks $A_{ij}$, the \emph{unimodularity condition} of $D$ requires that $A_{ii}+A_{ii}'$ ($i=1, \cdots, n$) and $D+D'$ are unimodular. It's shown in \cite{Liang} that there is a boundary simple $(2q-1)$-link $L$ whose Seifert matrix is $D=(A_{ij})$ when $q \ge 3$ \cite[Theorem 1]{Liang}.

Given a link $f \colon \cup_{i=1}^{n} S^{3} \hookrightarrow S^{5}$ we note that up to isotopy there is a unique tubular neighborhood $U$ of $\mathrm{Image}(f)$. We denote the complement of the interior of this tubular neighborhood by $X_f$ and use the tubular neighborhood to identify $\partial X_f$ with $\cup_{n}(S^1 \times S^3)$.

If two links $f \colon \cup_{i=1}^{n} S^{3} \hookrightarrow S^{5}$ and $f' \colon \cup_{i=1}^{n} S^{3} \hookrightarrow S^{5}$ are isotopic the isotopy extension theorem implies that the identification $\partial X_f \to \partial X_{f'}$ extends to a diffeomorphism $X_f \to X_{f'}$. In turn
 if there is an orientation-preserving diffeomorphism $g: X_{f} \to X_{f'}$ extending the identification on the boundary, then we can extend this by the identification on the tubular neighborhoods to an orientation-preserving diffeomorphism $\hat g: S^5 \to S^5$ mapping the first link to the second. Now we use the fact that $\pi_0 (\mathrm{Diff}^{+}(S^5))$ is isomorphic to the group of homotopy $6$-spheres (using the $h$-cobordism theorem and Cerf's theorem \cite{Cerf} that pseudo-isotopy implies isotopy) and that the group of $6$-dimensional homotopy spheres is trivial \cite{KM}. Thus the two links are isotopic.

Now note that link complement $X$ has free fundamental group of rank $n$, generated by the circles in the boundary components. Furthermore, from  Farber \cite[Theorem 5.7] {Farber} we know that $\pi_{2}$ of the complement of a simple boundary link is torsion free. Thus Theorem \ref{main1} applies to complements of simple boundary $3$-links in $S^5$.

The meridians give rise to an identification $\pi_{1}(X_{f}) \stackrel{\cong}{\rightarrow} F_{n}$, under this identification,  by the reinterpretation of the invariants in the beginning of section 3, we have an invariant
$$\psi(X_{f})=(\pi_2 (X_{f}),
b_{i} \colon  \pi_2(X_{f})^{*} \times \pi_2(X_{f})^{*} \to \Q, i=1, \cdots,n).$$
Here we consider $\pi_2(X_{f})$ as a $F_{n}$-module and $*$ stands for the $\mathbb Q$-dual.  The link complement$X_{f}$ is a Spin-manifold, thus Theorem \ref{main1} implies that this invariant determines the oriented diffeomorphism type mod boundary, meaning that the identification on the boundary extends to an orientation-preserving diffeomorphism between the whole manifolds. Thus we have proved the following

\begin{lemma}\label{lemma: link}
Two simple boundary $3$-links $f \colon \cup_{n}S^{3} \hookrightarrow S^{5}$ and $f' \colon \cup_{n}S^{3} \hookrightarrow S^{5}$ are isotopic if and only if under certain identifications of $\pi_{1}(X_{f})$ and $\pi_{1}(X_{f'})$ with $F_{n}$ coming from enumerating of the link components, $\psi(X_{f})$ and $\psi(X_{f'})$ are isomorphic.
\end{lemma}

\begin{proof}[Proof of Theorem \ref{theorem: link}]
By Lemma \ref{lemma: link}, to prove that the $l$-equivalence class of the Seifert matrices determines the isotopy type of the link, we need to show that the $l$-equivalence class of the Seifert matrices determines $\psi(X_{f})$. Let $F_i$ be Seifert manifolds of a link given by an embedding $f$. Let $X_f$ be the complement of the tubular neighborhood of the link, then the universal cover $\w X_f$ is obtained by glueing infinitely many copies of $Y$ via the deck transformation, where $Y$ is obtained from $X_f$ by cutting up along the Seifert manifolds. We identify $\pi_{1}(X_{f})$ with $F_{n}$ by sending the meridian (with the induced orientation from that of $S^{5}$ and $S^{3}$) of the $i$-th component of the link to the $i$-th standard generator $t_{i}$ of $F_{n}$. The Mayer-Vietoris sequence computing $H_{2}(\w X_f)$ is
$$\oplus_{i=1}^{n}H_{2}(F_{i}) \otimes_{\z} \z[F_{n}] \stackrel{\varphi}{\rightarrow} H_{2}(Y) \otimes_{\z} \z[F_{n}] \to H_{2}(\w X_f) \to 0$$
where under the basis of $H_{2}(F_{i})$ and the Alexander dual basis of $H_{2}(Y)$, and  $\varphi$ is given by $(A_{ij} - t_{i}A_{ij}')$. Therefore the $\z[F_{n}]$-module $H_{2}(\w X_{f})$ is determined by $D=(A_{ij})$. Also we see that the map $H_{2}(F_{i}) \to H_{2}(\w X_{f})$ is determined by $D=(A_{ij})$, hence the dual map $H^{2}(\w X_{f};\Q) \to H^{2}(F_{i};\Q)$. And the intersection form of $F_{i}$ is given by $A_{ii}+A_{ii}'$. It's easy to see from the definition that the bilinear pairing $b_{i}$ is given by the composition $H^{2}(\w X_{f};\Q) \to H^{2}(F_{i}; \Q)$ with the intersection form on $H^{2}(F_{i};\Q)$. Therefore the bilinear form $b_{i}$ is determined by the Seifert matrix $D=(A_{ij})$.

Given two simple boundary $3$-links $L_{0}$ and $L_{1}$, with $l$-equivalent Seifert matrices $D_{0}=(A_{ij}^{(0)})$ and $D_{1}=(A_{ij}^{(1)})$, then by \cite[Lemma 1]{Liang} we may choose Seifert manifolds $\{F_{i}^{0}\}$ and $\{F_{i}^{1}\}$ of $L_{0}$ and $L_{1}$ respectively, such that the corresponding Seifert matrices are equal. Then by the above discussion $L_{0}$ and $L_{1}$ are equivalent.

\smallskip

Using a stabilization trick introduced by Levine in the case of knots, we can extend the construction of links with given Seifert matrix in \cite{Liang}  to the case $q=2$.  The construction goes as follows.

Firstly by \cite[Lemma 16]{Levine} we may find embeddings $F_{i}^{4} \subset B_{i}^{5} \subset S^{5}$ with $\partial F_{i}=S^{3}$ is a simple $3$-knot, whose Seifert matrix $A_{i}$ is S-equivalent to $A_{ii}$. After stabilization by connected sum with copies of $S^{2} \times S^{2}$, these Seifert manifolds $F_{i}^{4}$ are diffeomorphic to connected sums of $S^{2} \times S^{2}$ and the Kummer surface with a $4$-ball $B^{4}$ deleted. These smooth $4$-manifolds all have a handle decomposition of the form $F_{i}= D^{4} \cup h_{1} \cup \cdots \cup h_{k}$ where the $h_{i}$'s are $2$-handles (see e.~g.~\cite{Mand}). Then by the same argument in the proof of \cite[Theorem 1]{Liang} we can show that the new Seifert matrix $D'$, which is $l$-equivalent to $D$, is the Seifert matrix of a boundary simple $3$-link $L$.
\end{proof}

Now we describe the Milnor pairing associated to an $(m \times m)$ integral matrix $D=(A_{ij})$ satisfying the unimodularity conditions. Let $\varphi_{D} \colon (\z[F_{n}])^{m} \to (\z[F_{n}])^{m}$ be the $\z[F_{n}]$-module map given by the matrix $(A_{ij}-t_{i}A_{ij}')$.  Assume the square matrix $A_{ii}$ has dimension $m_{i}$, then $A_{ii}+A_{ii}'$ defines a symmetric bilinear form $I_{i}$ on $\z^{m_{i}}$. Let $\iota_{i}$ be the composition
$$\iota_{i} \colon \z^{m_{i}} \stackrel{\oplus_{j}A_{ij}}{\longrightarrow} \oplus_{j}\z^{m_{j}} \to \oplus_{j}\z^{m_{j}} \otimes_{\z} \z[F_{n}] =(\z[F_{n}])^{m} \to \mathrm{coker}\varphi_{D}$$
The $\mathbb Q$-dual of $\iota_{i}$ is $\iota_{i}^{*} \colon (\mathrm{coker}\varphi_{D} )^{*} \to \mathbb Q^{m_{i}}$.
Let $C_{1}=(\z[F_{n}])^{n} \stackrel{d_{1}}{\rightarrow} C_{0}=\z[F_{n}] $ be the standard chain complex computing $H_{*}(BF_{n};\z[F_{n}])$, $\{ e_{i}, i=1, \cdots, n\} $ be standard basis of $(\z[F_{n}])^{n}$, $\{ e_{i}^{*}, i=1, \cdots, n\}$ be the dual basis, $[e_{i}^{*}] \in H^{1}(BF_{n};\z[F_{n}])$ be the corresponding cohomology class.  Then the bilinear form
$$b_{D} \colon (\mathrm{coker} \varphi_{D})^{*} \times (\mathrm{coker} \varphi_{D})^{*} \to (H^{1}(BF_{n};\mathbb Q[F_{n}]))^{*} $$
is given by $\langle b_{D}(u,v), [e_{i}^{*}] \rangle = I_{i}(\iota_{i}^{*}(u), \iota_{i}^{*}(v))$. (See Lemma \ref{form}.)

\begin{proof}[Proof of Theorem \ref{theorem: action}]
There is a surjective map from the set of isotopy classes of simple boundary $n$-components links  $L \subset S^{5}$ to the set of diffeomorphism classes of smooth oriented closed $5$-manifolds $M^{5}$ with free fundamental group of rank $n$, and $H_{2}(M;\z)=0$. This is given by surgery: given a link $L$, we may do surgery on $L$ and obtain a $5$-manifold $M$ with $H_{2}(M;\z)=0$. If $L$ is simple boundary, then it's easy to see that $\pi_{1}(M)$ is isomorphic to $F_{n}$. The meridians of the link components form an embedding $\cup_{n} S^{1} \subset M$, and these circles generate $\pi_{1}(M)$.  On the other hand, given such an  $M^{5}$ we may choose an embedding $\cup_{n} S^{1} \subset M^{5}$ such that the circle generate $\pi_{1}(M)$. Then we do surgery on this embedding and obtain $S^{5}$, the complementary spheres $\cup_{n} S^{3} \subset S^{5}$ form a link $L$. Clearly this is a simple boundary link.

By comparing the definitions, we see that the generalized Milnor pairing $\varphi (M)$ of $M$ is the same as the generalized Milnor pairing $\psi(X_{f})$ of the link complement defined before Lemma \ref{lemma: link}. In the proof of Theorem \ref{theorem: link} we have shown how the generalized Milnor pairing $\psi(X_{f})$ is determined by the Seifert matrix $D=(A_{ij})$. This is exactly $(F_{n}, \mathrm{coker}\varphi_{D}, b_{D})$, which was described before the proof of Theorem \ref{theorem: action}. By Theorem \ref{theorem: link}, all such matrices satisfying the unimodular conditions are realized by simple boundary links. This finishes the proof.
\end{proof}

\section{Appendix}
In this appendix we show some basic properties of the class of manifolds mentioned in Remark 1, i.~e.~oriented closed $5$-manifolds $M$ with $\pi_{1}(M) \cong \z$ and $\pi_{2}(M)$ a finitely generated abelian group.

\begin{lemma}\label{finite}
Let $M^5$ be a $5$-manifold with $\pi_1(M)=\mathbb Z$ and
$\pi_2(M)$ a finitely generated abelian group, then all higher
homotopy groups $\pi_i(M)$ ($i \ge 2$) are finitely generated
abelian groups.
\end{lemma}
\begin{proof}
By Serre's mod $\mathscr C$ theory \cite{Serre}, we only need to show that $H_i(\widetilde M)$ ($i \ge 3$) are finitely generated abelian groups. The only problem is $H_3(\widetilde M)$. $H_3(\widetilde M)=H_3(M;\Lambda)\cong H^2(M;\Lambda)$, where $\Lambda = \mathbb Z[\mathbb Z]=\mathbb Z[t,t^{-1}]$ is the group ring. By \cite[Proposition 3.3]{Wall1}, the CW-structure of $M$ has the form
$$M=S^1 \vee (\vee S^2) \cup \cdots$$
Therefore the cellular chain complex $C_*(M;\Lambda)$ has the form
$$\cdots \to C_3 \stackrel{d}{\rightarrow} C_2 \stackrel{0}{\rightarrow} C_1 \to C_0.$$
From the exact sequence $C_3 \stackrel{d}{\rightarrow} C_2 \to \mathrm{coker}d \to 0$ we have the dual exact sequence
$0 \to (\mathrm{coker}d)^* \to C_2^* \stackrel{d^*}{\rightarrow} C_3^*$, hence
$H^2(M;\Lambda)=\ker d^*=(\mathrm{coker}d)^*$. Now $\mathrm{coker}d=H_2(M;\Lambda)=\pi_2(M)$ is a finitely generated abelian group, the proof is done by the following lemma.
\end{proof}

\begin{lemma}\label{lem:algebra}
If a $\Lambda$-module $G$ is a finitely generated abelian group, then $\Hom_{\Lambda}(G,\Lambda)=0$.
\end{lemma}
\begin{proof}
The torsion subgroup $T$ is a sub-$\Lambda$-module, the exact sequence $0 \to T \to G \to G/T \to 0$ induces an exact sequence
$$0 \to \Hom_{\Lambda}(G/T, \Lambda) \to  \Hom_{\Lambda}(G, \Lambda) \to  \Hom_{\Lambda}(T, \Lambda)$$
therefore $\Hom_{\Lambda}(G/T, \Lambda) \cong \Hom_{\Lambda}(G, \Lambda)$ since $\Hom_{\Lambda}(T, \Lambda)=0$. Therefore we may assume that $G$ is a finitely generated free abelian group.

Let $x_{1}, \cdots, x_{n}$ be a basis of $G$, a $\Lambda$-module structure on $G$ is given by $A\in GL_{n}(\mathbb Z)$, which specifies the action of the generator $t$ on the basis.  A $\Lambda$-homomorphism $G \to \Lambda$ is given by $v_{1}, \cdots, v_{n} \in \Lambda$ which are the images of $x_{1}, \cdots, x_{n}$. The $n$-tuple $v=(v_{1}, \cdots ,v_{n})$ should satisfy the equation $(tI-A)v=0$. Clearly $\mathrm{det}(tI-A)\ne 0$, thus the equation has no non-zero solution in the quotient field ($\Lambda$ is an integral domain), hence also has no non-zero solution in $\Lambda$. Therefore $\Hom_{\Lambda}(G,\Lambda)=0$.
\end{proof}

Now let $M^5$ be a closed orientable $5$-manifold with $\pi_1(M)=\mathbb Z$ and $\pi_2(M)$ a finitely generated abelian group. Fix an orientation of $M$ and a generator $t$ of $\pi_1(M)$, these choices determine a generator (a fundamental class) $\sigma_{M} \in H_4(\widetilde M) =\mathbb Z$. Then on the finitely
generated  free abelian group $H^2(\widetilde M)$ a symmetric
bilinear form $H^2(\widetilde M) \times H^2(\widetilde M) \to
\mathbb Z$ is defined by $(\alpha, \beta) \mapsto \langle \alpha
\cup \beta , \sigma_{M} \rangle$. By the following proposition we see that this bilinear form is
unimodular and $\pi_2(M)$ is a free abelian group. Thus this
bilinear form induces a symmetric bilinear form on
$\pi_2(M)=\pi_2(\widetilde M)=H_2(\widetilde M)=H^2(\widetilde
M)^*$, denoted by $I(M)$.

\begin{proposition}\label{prop:1}
Let $M^5$ be an orientable $5$-manifold with $\pi_1(M)=\mathbb Z$ and
$\pi_2(M)$ a finitely generated abelian group. Then we have the
following
\begin{enumerate}
\item $\pi_2(M)$ is torsion free.
\item  The symmetric bilinear form $I(M)$ is unimodular; $I(M)$ is even if and only if $w_{2}(M)=0$.
\item $\langle p_1(M), \sigma_{M} \rangle=3 \cdot \mathrm{sign}(I(M))$, where $\sigma_{M} \in H_4(\widetilde M)$ is the generator determined by the orientation of $M$ and the generator $t$ of $\pi_1(M)$.
\end{enumerate}
\end{proposition}

\begin{proof}
Consider $M \times \mathbb C \mathrm P^{2}$. By the Lemma \ref{finite}
and Browder-Levine's fibration theorem \cite{Browder-Levine}, we know that this manifold is a fiber bundle
over $S^{1}$ with simply-connected fiber $F^{8}$. Therefore
$\widetilde M \times \mathbb C \mathrm P^{2}$ is homotopy equivalent
to $F$.

(1) By K\"unneth formula and Poincar\'e duality, we have
$$H^{3}(\widetilde M) \cong H^{7}(\widetilde M \times \mathbb C \mathrm P^{2}) \cong H^{7}(F) \cong H_{1}(F)=0.$$
This proves that
$\mathrm{tors}\pi_{2}(M)=\mathrm{tors}H_{2}(\widetilde
M)=\mathrm{tors}H^{3}(\widetilde M)=0$.

(2) On $H^{4}(\widetilde M \times \mathbb C \mathrm P^{2})$ there is
defined a symmetric bilinear form $I(M \times \mathbb C \mathrm
P^{2})$, which is isometric to the tensor product of $I(M)$ and the
intersection form of $\mathbb C \mathrm P^{2}$ plus a hyperbolic form. On the other hand,
 the bilinear form $I(M \times \mathbb
C \mathrm P^{2})$  is isometric to  the intersection form of
$F$, which is unimodular by Poincar\'e duality. Therefore the
bilinear form $I(M)$ is unimodular.

From the discussion above we see that $I(M)$ is even if and only if
the Wu class $v_4(F)=0$. The Wu classes and Stiefel-Whitney classes
of $M$ and $F$ are related as follows. Let $i \colon F \to M \times \mathbb C \mathrm P^{2}$ be the inclusion of the fiber, then $TF \oplus \underline{\mathbb R}= i^{*}T(M \times \mathbb C \mathrm P^{2})$. We have
$$w_{2}(M)=v_{2}(M), \ w_{3}(M)=Sq^{1}w_{2}(M), \ w_{4}(M)=w_{2}(M)^{2}.$$
$$v_{2}(F)=w_{2}(F)=i^{*}(w_{2}(M)+w_{2}(\mathbb C \mathrm P^{2})),$$
$$w_{3}(F)=Sq^{1}w_{2}(F)+v_{3}(F),$$ on the other hand, $w_{3}(F)=i^{*}w_{3}(M)$,
from this we have
$$v_{3}(F)=i^{*}(Sq^{1}w_{2}(M)+w_{3}(M)).$$
By the Wu formula
$$w_{4}(F)=v_{2}(F)^{2}+Sq^{1}v_{3}(F)+v_{4}(F),$$
on the other hand
$$w_{4}(F)=i^{*}(w_{4}(M)+w_{2}(M)w_{2}(\mathbb C \mathrm P^{2})+w_{4}(\mathbb C \mathrm P^{2})),$$ compare these two equations we have
$$v_{4}(F)=i^{*}(w_{2}(M)w_{2}(\mathbb C \mathrm P^{2})).$$
But $H^{3}(F;\mathbb Z_{2})\cong H^{3}(\widetilde M \times \mathbb C \mathrm P^{2};\mathbb Z_{2}) \cong H^{3}(\widetilde M;\mathbb Z_{2})=0$
(the last identity is a consequence of the fact that $H_{2}(\widetilde M)$ is free and $H_{3}(\widetilde M)=0$, see Lemma \ref{finite}),
from the Wang sequence we see that $i^{*} \colon H^{4}(M \times \mathbb C \mathrm P^{2};\mathbb Z_{2}) \to H^{4}(F;\mathbb Z_{2})$ is injective.
Thus $v_{4}(F)=0$ if and only if $w_{2}(M)=0$.

(3) Since $I(M)$ and $I(M \times \mathbb C \mathrm P^{2})$ differ by a hyperbolic form, we have
$$\mathrm{sign}(I(M))=\mathrm{sign}(I(M \times \mathbb C \mathrm P^{2}))=\mathrm{sign}(F)=\frac{\langle 7p_{2}(F)-p_{1}(F)^{2}, [F]\rangle}{45},$$ where the last identity is the Hirzebruch index formula.  Since $F$ has trivial normal bundle in $M \times \mathbb C
\mathrm P^{2}$, we have
$$p_{1}(F)=i^{*}p_{1}(M \times \mathbb C \mathrm P^{2})=i^{*}(p_{1}(M)+p_{1}(\mathbb C \mathrm P^{2})),$$
$$p_{2}(F)=i^{*}p_{2}(M \times \mathbb C \mathrm P^{2})=i^{*}(p_{1}(M)p_{1}(\mathbb C \mathrm P^{2})).$$
A straightforward calculation shows that $3\cdot
\mathrm{sign}(I(M))=\langle p_{1}(M), \sigma_{M}\rangle$.
\end{proof}

\end{document}